\RequirePackage[l2tabu,orthodox]{nag}
\documentclass{amsart}
\usepackage{etex}
\usepackage[T1]{fontenc}
\usepackage[utf8]{inputenx}
\usepackage{lmodern}
\usepackage[kerning=true,tracking=true]{microtype}

\title{Summary of progress on the Blaschke conjecture}
\author{Benjamin McKay}
\email{b.mckay@ucc.ie}
\thanks{Thanks to Werner Ballmann, Karsten Grove, Linus Kramer and Juan-Carlos Alvarez Paiva for their help.}
\usepackage{varwidth}
\usepackage{array}
\usepackage{booktabs}
\usepackage{mathtools}
\usepackage{amssymb}
\usepackage{amsthm}
\usepackage[usenames,dvipsnames]{xcolor}
\usepackage{tikz}
\usepackage{tikz-cd}
\usetikzlibrary{intersections}
\usepackage{varioref}\vrefwarning
\usepackage{braket}
\usepackage{mathrsfs}
\usepackage{pdftexcmds}

\usepackage[pdftex,pagebackref]{hyperref}
\hypersetup{
  colorlinks   = true,  
  urlcolor     = black, 
  linkcolor    = black, 
  citecolor    = black  
}
\mathtoolsset{centercolon}
\newcommand*{\mathlabel}[1]{\footnotesize\(#1\)}
\newcommand*{\of}[1]{\ensuremath{\!\left(#1\right)}}
\newcommand*{\K}[2]{\ensuremath{\mathbb{#1}^{#2}}}
\newcommand*{\C}[1]{\K{C}{#1}}
\newcommand*{\R}[1]{\K{R}{#1}}
\newcommand*{\Ha}[1]{\K{H}{#1}}
\newcommand*{\Z}[1]{\K{Z}{#1}}
\newcommand*{\Oc}[1]{\K{O}{#1}}
\newcommand*{\KP}[2]{\ensuremath{\mathbb{{#1}P}^{#2}}}
\newcommand*{\CP}[1]{\ensuremath{\KP{C}{#1}}}
\newcommand*{\HP}[1]{\ensuremath{\KP{H}{#1}}}
\newcommand*{\OP}[1]{\ensuremath{\KP{O}{#1}}}
\newcommand*{\RP}[1]{\ensuremath{\KP{R}{#1}}}
\newcommand*{\Cohom}[2]{\ensuremath{H^{#1}\!\left(#2\right)}}
\newcommand*{\cut}[1]{\ensuremath{#1^{\perp}}}
\newcommand*{\Gr}[2]{\ensuremath{\operatorname{Gr}_{\mathbb{R}}\!\left({#1},{#2}\right)}}
\newcommand*{\pr}[1]{\ensuremath{\left({#1}\right)}}

\newcommand*{\geod}[1]{\ensuremath{\operatorname{Geod}\!\pr{#1}}}
\newcommand*{\vol}[1]{\ensuremath{\operatorname{Vol} #1}}
\newcommand*{\area}[1]{\ensuremath{\operatorname{Area} #1}}
\newcommand*{\inj}[2][]{\ensuremath{\operatorname{inj}^{#1}\! #2}}

\newcommand*{\GL}[1]{\ensuremath{\operatorname{GL}\!\pr{#1}}}
\makeatletter
\newcommand*{\ev}[3][.]%
{%
\ifnum\pdf@strcmp{#1}{.}=\z@ 
\ensuremath{\operatorname{ev}_{#2}}%
\else%
\ensuremath{\operatorname{ev}#3_{#1}\of{#2}}%
\fi%
}
\makeatother
\makeatletter
\newcommand*{\pd}[3][1]
{
\frac{%
\partial%
\ifnum\pdf@strcmp{#1}{1}=0\else^{#1}\fi%
#2}%
{%
\partial
#3
\ifnum\pdf@strcmp{#1}{1}=0\else^{#1}\fi%
}%
}
\makeatother
\newcommand*{\defeq}{\mathrel{\vcenter{\baselineskip0.5ex \lineskiplimit0pt
                     \hbox{\scriptsize.}\hbox{\scriptsize.}}}%
                     =}
                     
\newcommand*{\scalar}[1]{\ensuremath{\operatorname{scal}_{#1}}}
\newcommand*{\Ricci}[1]{\ensuremath{\operatorname{Ric}\of{#1}}}
\newtheorem{theorem}{Theorem}
\newtheorem{corollary}{Corollary}
\newtheorem{lemma}{Lemma}
\newtheorem{proposition}{Proposition}
\theoremstyle{remark}

\newtheorem{remark}{Remark}

\makeatletter
\newcommand*{\drawpoint}[4][.]{
	\filldraw[gray!50!white] ({#3},{#4}) circle (1pt)
\ifnum\pdf@strcmp{#1}{.}=\z@ 
\else
	node[black,opacity=1,#2] {\mathlabel{#1}} 
\fi
;
}
\makeatother
\newcommand*{\curvecolour}{gray!50}
\newcommand*{\mediumcurvecolour}{gray!30}
\newcommand*{\lightcurvecolour}{gray!20}
\newcommand*{\verylightcurvecolour}{gray!10}

\definecolor{example-color}{RGB}{255,255,255}
\definecolor{example-border-color}{RGB}{200,200,200}

\makeatletter
\def\@tocline#1#2#3#4#5#6#7{\relax
  \ifnum #1>\c@tocdepth 
  \else
    \par \addpenalty\@secpenalty\addvspace{#2}%
    \begingroup \hyphenpenalty\@M
    \@ifempty{#4}{%
      \@tempdima\csname r@tocindent\number#1\endcsname\relax
    }{%
      \@tempdima#4\relax
    }%
    \parindent\z@ \leftskip#3\relax \advance\leftskip\@tempdima\relax
    #5\leavevmode\hskip-\@tempdima #6\nobreak\relax
    ,~#7\par
    \endgroup
  \fi}
\makeatother

\usepackage{colortbl}
\arrayrulecolor{gray!30}
\makeatletter
    \def\rulecolor#1#{\CT@arc{#1}}
    \def\CT@arc#1#2{%
      \ifdim\baselineskip=\z@\noalign\fi
      {\gdef\CT@arc@{\color#1{#2}}}}
    \let\CT@arc@\relax
\rulecolor{gray!30}
\makeatother

\begin{document}
\maketitle
\begin{center}
\begin{varwidth}{\textwidth}
\tableofcontents
\end{varwidth}
\end{center}
\section{Definition}
In 1883, over a region of more than \(\frac{1}{3}\) of the Earth's surface, people could hear the explosion of the volcano Krakatau.
Citizens of Bogota, Columbia heard the explosion 7 times, as the sound focused and defocused at two antipodal points.
All of the geodesics leaving any point of a sphere minimize distance until they reach ``all the way across'' the sphere, where they simultaneously collide; see \cite{Arnold:2001} for applications.
A point \(p\) of a compact metric space is a \emph{Blaschke point} if every geodesic (i.e. locally shortest path) leaving \(p\) and of length less than the diameter is the unique shortest path between any of its points.
A compact connected metric space is \emph{Blaschke} if its every point is Blaschke.
\section{Examples}
Clearly the round sphere is Blaschke.
Each geodesic of \(\CP{2}\) lies inside a unique totally geodesic complex projective line \(\CP{1} \subset \CP{2}\).
Our geodesic's competitors for minimizing distance also lie in the same \(\CP{1}\), so our geodesic stops minimizing just when it reaches the antipode in \(\CP{1}\) of its starting point, a maximal distance.
Therefore \(\CP{2}\) is Blaschke, and likewise \(\CP{n}\), \(\HP{n}\) and \(\OP{2}\) are Blaschke.
Picture \(\RP{2}\): each
geodesic from the north pole travels to the equator, where it disappears and reappears on the opposite side of the equator, heading north.
Each geodesic minimizes distance up to the equator, which is as far as you can go from the north pole.
For the same reason, \(\RP{n}\) is Blaschke.
The Blaschke conjecture (as understood today) is the conjecture that the only Blaschke manifolds are \(\RP{n}, S^n, \CP{n}, \HP{n}\) and \(\OP{2}\) (i.e. the compact rank one symmetric spaces) with their standard metrics, up to rescaling each metric by a positive constant.

Imagine that you can see any star in the night sky, as its light reaches you along a geodesic.
But suppose that if that geodesic travels further than the diameter of your universe the light becomes too dim and so invisible.
If your universe is \(S^3\), then you can see any star as a single point of light in the sky, unless the star sits at your antipode, when you see the star as if it were a sphere of fire in the sky, surrounding you in all directions.
If your universe is \(\RP{3}\), you see each star as single point of light, unless it lies on your cut locus \(\RP{2}\), when you see the star as if it were two points of light, in opposite directions in the night sky.
If your universe is \(\CP{2}\), you see each star as a single point of light, unless it lies on your cut locus \(\CP{1}\), i.e. lies at the antipode of a round \(\CP{1}\) through your location, when you see the star as if it were a ring of fire circling you: a fiber of the Hopf fibration.
The same phenomena arise in all Blaschke manifolds, as we will see, and provide the only tool discovered so far to understand their topology.

Picture a star at a point \(p_0 \in \CP{2}\) exploding into a burst of light; the burst evolves as the sphere of radius \(t\) about that point in the Riemannian metric.
The sphere starts off its life as a small, nearly Euclidean, 3-dimensional sphere.
As the sphere evolves, the circles of the Hopf fibration (i.e. its intersections with complex projective lines through \(p_0\)) at first grow, and then later shrink, while in perpendicular directions the sphere grows steadily.
The sphere collapses onto a Euclidean 2-sphere, the cut locus of \(p_0\), consisting of the points at maximal distance from \(p_0\), as each circle of the Hopf fibration contracts to a point.
Similar phenomena occur in the other projective spaces, and in the Blaschke manifolds.

A flat torus
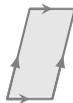
\[
\begin{tikzpicture}[scale=.3]
\fill[thick,gray!20,draw=gray] (0,0) -- (2,0) -- (3,4) -- (1,4) -- cycle;
\draw[thick,gray,-stealth] (0,0) -- (1,0);
\draw[thick,gray,-stealth] (1,4) -- (2,4);
\draw[thick,gray,-latex] (0,0) -- (.5,2);
\draw[thick,gray,-latex] (2,0) -- (2.5,2);
\end{tikzpicture}
\]
has diameter
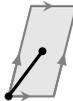
\[
\begin{tikzpicture}[scale=.3]
\fill[thick,gray!20,draw=gray] (0,0) -- (2,0) -- (3,4) -- (1,4) -- cycle;
\draw[thick,gray,-stealth] (0,0) -- (1,0);
\draw[thick,gray,-stealth] (1,4) -- (2,4);
\draw[thick,gray,-latex] (0,0) -- (.5,2);
\draw[thick,gray,-latex] (2,0) -- (2.5,2);
\draw[very thick,black] (0,0) -- (1.5,2);
\fill[black] (1.5,2) circle (5pt);
\fill[black] (0,0) circle (5pt);
\end{tikzpicture}
\]
and shortest minimal geodesic
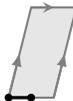
\[
\begin{tikzpicture}[scale=.3]
\fill[thick,gray!20,draw=gray] (0,0) -- (2,0) -- (3,4) -- (1,4) -- cycle;
\draw[thick,gray,-stealth] (1,4) -- (2,4);
\draw[thick,gray,-latex] (0,0) -- (.5,2);
\draw[thick,gray,-latex] (2,0) -- (2.5,2);
\draw[very thick,black] (0,0) -- (1,0);
\fill[black] (1,0) circle (5pt);
\fill[black] (0,0) circle (5pt);
\end{tikzpicture}
\]
from which it is easy to see that no flat torus of any dimension is Blaschke.

\section{What we know}
As we will see below, any Blaschke manifold has the integral cohomology ring of a unique compact rank 1 symmetric space, called its \emph{model} \cite{Bott:1954,Samelson:1963}.
The Blaschke conjecture is the claim that each Blaschke manifold is isometric to its model, up to constant metric rescaling. 
Today we know that, up to constant metric rescaling, every Blaschke manifold \(M\) is \emph{adjective} to its model:\label{table:know}
\begin{center}
\renewcommand{\arraystretch}{1.3}
\tabcolsep=0.08cm
\begin{tabular}{@{}>{\(}l<{\)}ll@{}}
\toprule
\text{model} & adjective & proof \\ \midrule
\RP{n} & isometric & \cite{Besse:1978,Green:1963,Weinstein:1974,Yang:1980} \\
S^n & isometric & \cite{Besse:1978,Green:1963,Weinstein:1974,Yang:1980} \\
\CP{2} & diffeomorphic & \cite{McKay:2005} \\
\HP{2} & diffeomorphic &  \cite{Kramer/Stolz:2007,Reznikov:1994} \\
\OP{2} & diffeomorphic &  \cite{Kramer/Stolz:2007,Reznikov:1994} \\
\HP{n} & homotopy equivalent & \cite{Sato:1985} \\
\bottomrule
\end{tabular}
\end{center}
As we can see in the table, the Blaschke conjecture is solved just for homology spheres and homology real projective spaces.
The proofs of my own paper \cite{McKay:2004} and of Yang's \cite{Yang:1990,Yang:1993} are flawed; see the section~\ref{section:sphere.fibrations} below for more information.
These flaws have altered the above table from the table that appeared in the published version of the paper you are currently reading.

\section{Equivalent conditions}
Each of the following is a necessary and sufficient condition that a connected compact Riemannian manifold \(M\) be Blaschke; see below and \cite{Besse:1978} chapter 5.
\begin{enumerate}
\item The diameter of \(M\) equals the injectivity radius of \(M\).
\item The distance of any point to its cut locus is the same for all points.
\item The cut locus of any point is a metric sphere, and the radius of that sphere is the same for all points.
\item All geodesics are simply closed loops and all cut loci are metric spheres.
\item If \(p\) and \(q\) are points of \(M\) and \(q\) lies in the cut locus of \(p\), then the unit tangent vectors to minimal geodesics heading from \(p\) to \(q\) form a great sphere in the unit sphere in \(T_p M\).
\item If \(D\) is the diameter of \(M\) then the exponential map \(\exp \colon T_p M \to M\) at any point \(p \in M\) is a diffeomorphism on the open ball of radius \(D\) and a fiber bundle mapping on the sphere of radius \(D\), with fibers great spheres.
\end{enumerate}

\section{History}
In the first edition of \cite{Blaschke:1921} (\textsection{86} p. 155 question 2), Blaschke  conjectured that any surface in \(\R{3}\) whose every point has a unique conjugate point is isometric to a round sphere.
In the second edition, Reidemeister gave a proof in an appendix, trying to show that a surface with the structure of a Blaschke metric space is a Desargues projective plane \cite{Salzmann:1995}. 
The third edition exposed an error in that proof; see \cite{Busemann:1957} for a simple construction of non-Desargues projective planes with metrics so that the projective lines are geodesics.
Green \cite{Green:1963} proved Blaschke's conjecture using elementary classical surface geometry.
Besse \cite{Besse:1978} expanded the Blaschke conjecture to the one stated above, and (as will see) Berger, Kazdan, Weinstein and Yang proved the conjecture for homology spheres and homology real projective spaces.
Subsequent authors have only considered the diffeomorphism types of Blaschke manifolds (as in our table), discovering nothing about the metric geometry, about which little is known.

\section{Cut locus geometry}

\begin{proposition}\label{proposition:cut}
Suppose that \(m_0 \in M\) is a point in a Blaschke manifold and \(\cut{m_0}\) is the cut locus of \(m_0\).
Then at every cut point \(c \in \cut{m_0}\), every tangent vector to \(M\) is uniquely expressed as a sum of a tangent vector to the cut locus and a normal vector to the cut locus.
Each normal vector to the cut locus is the tangent vector to a minimal geodesic from \(m_0\).
\end{proposition}
\begin{remark}
A vector \(v \in T_m M\) is \emph{tangent} to a set \(S \subset M\) if \(v\) is the velocity of a continuously differentiable curve, defined on an open interval, with image in \(S\). The vectors tangent to \(S\) at a point \(s \in S\) form the \emph{tangent space} \(T_s S\).
Reverse or reparameterize: \(T_s S \subset T_s M\) is closed under rescaling.
\end{remark}
\begin{proof}
First, we want to see that the tangent spaces of the cut locus \(\cut{m_0}\) are perpendicular to the minimal geodesics from \(m_0\).
Since \(\cut{m_0}\) is the set of points furthest from \(m_0\), every path heading ``away'' from \(\cut{m_0}\) moves closer to \(m_0\).
Take any minimal geodesic \(p(t)\) heading from \(m_0\) to a point \(c \in \cut{m_0}\); suppose that \(p(0)=m_0\).
If \(D\) is the diameter of \(M\), \(p(D)=c\).
Let \(w=p'(D)\).
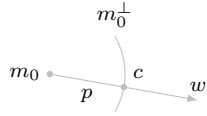
\par\noindent%
\begin{center}
\begin{tikzpicture}
\draw[\curvecolour] ({cos(-30)},{sin(-30)}) arc (-30:30:1) node[above,black] {\mathlabel{\cut{m_0}}};
\draw[\curvecolour] (0,0) -- ({cos(-10)},{sin(-10)}) node[xshift=-14,yshift=-3,black]{\mathlabel{p}};
\draw[\curvecolour,-latex] ({cos(-10)},{sin(-10)}) -- ({2*cos(-10)},{2*sin(-10)}) node[above,black]{\mathlabel{w}};
\drawpoint[m_0]{left}{0}{0}
\drawpoint[c]{above right}{cos(-10)}{sin(-10)}
\end{tikzpicture}
\end{center}
\par\noindent%
For any vector \(v \in T_c M\), if \(w \cdot v < 0\) then travelling along any path with velocity \(v\) will move us closer to \(m_0\).
To see this, let \(m_{\varepsilon}=p(\varepsilon)\), so that the distance between \(m_{\varepsilon}\) and \(c\) is \(D-\varepsilon\).
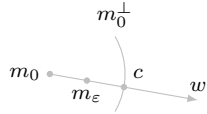
\par\noindent%
\begin{center}
\begin{tikzpicture}
\draw[\curvecolour] ({cos(-30)},{sin(-30)}) arc (-30:30:1) node[above,black] {\mathlabel{\cut{m_0}}};
\draw[\curvecolour] (0,0) -- ({cos(-10)},{sin(-10)});
\draw[\curvecolour,-latex] ({cos(-10)},{sin(-10)}) -- ({2*cos(-10)},{2*sin(-10)}) node[above,black]{\mathlabel{w}};
\drawpoint[m_0]{left}{0}{0}
\drawpoint[c]{above right}{cos(-10)}{sin(-10)}
\drawpoint[m_{\varepsilon}]{below}{.5*cos(-10)}{.5*sin(-10)}
\end{tikzpicture}
\end{center}
\par\noindent%
For \(\varepsilon\) with \(0 < \varepsilon < D\), the gradient of distance from \(m_{\varepsilon}\) at \(c\) is \(w\), so moving along a path with velocity \(v\) moves us strictly closer to \(m_{\varepsilon}\), at a rate \(w \cdot v\) independent of \(\varepsilon\); let \(\varepsilon \to 0\).
Therefore if \(v\) is a tangent vector \(v \in T_c \cut{m_0}\) to the cut locus, then \(w \cdot v \ge 0\), and since \(\pm v \in T_c \cut{m_0}\), \(w \cdot v=0\).

Pick any vector \(u \in T_c M\) not proportional to \(w\) so that \(u \cdot w < 0\); in particular \(u\) is not tangent to \(\cut{m_0}\).
\par\noindent%
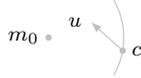
\begin{center}
\begin{tikzpicture}
\draw[\curvecolour] ({cos(30)},{sin(30)}) arc (30:-30:1);
\draw[-latex,\curvecolour] ({cos(-10)},{sin(-10)}) -- ({.6*cos(20)},{.6*sin(20)}) node[left,black]{\mathlabel{u}};
\drawpoint[m_0]{left}{0}{0}
\drawpoint[c]{right}{cos(-10)}{sin(-10)}
\end{tikzpicture}
\end{center}
\par\noindent%
Let \(q(t) \defeq \exp_c\of{tu}\).
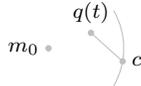
\par\noindent%
\begin{center}
\begin{tikzpicture}
\draw[\curvecolour] ({cos(30)},{sin(30)}) arc (30:-30:1);
\draw[\curvecolour] ({.6*cos(20)},{.6*sin(20)}) -- ({cos(-10)},{sin(-10)});
\drawpoint[m_0]{left}{0}{0}
\drawpoint[c]{right}{cos(-10)}{sin(-10)}
\drawpoint[q(t)]{above}{.6*cos(20)}{.6*sin(20)}
\end{tikzpicture}
\end{center}
\par\noindent%
Because \(u \cdot w < 0\), the distance to \(m_0\) along \(q(t)\) decreases for small \(t\), and so \(q(t)\) stays away from the cut locus \(m_0^{\perp}\) for small \(t>0\).
We travel along the unique minimal geodesic from \(m_0\) out to \(q(t)\); as we vary \(t\), we spread out a family of minimal geodesics, forming a surface in \(M\).
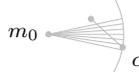
\par\noindent%
\begin{center}
\begin{tikzpicture}
\draw[\curvecolour] ({cos(30)},{sin(30)}) arc (30:-30:1);
\foreach \i in {-12,-8,...,12}{
	\draw[\curvecolour,very thin] (0,0) -- ({cos(\i)},{sin(\i)});
}
\draw[\curvecolour] ({.6*cos(20)},{.6*sin(20)}) -- ({cos(-10)},{sin(-10)});
\drawpoint[m_0]{left}{0}{0}
\drawpoint{}{.6*cos(20)}{.6*sin(20)}
\drawpoint[c]{below right}{cos(-10)}{sin(-10)}
\end{tikzpicture}
\end{center}
Along that surface, we let \(w\) be the unit velocity of the geodesic rays, i.e. \(w=\partial_s\), with \(s\) the arclength along each ray measured from \(m_0\).
Let \(v\) be the projection of \(\partial_t\) to \(w^{\perp}\).

Let's worry about how smooth this surface is along the edges \(s=D\) and \(t=0\).
Construct the same surface, but starting at \(m_{\varepsilon}\) instead of \(m_0\).
This surface is smooth out past \(t=0\) and \(s=D\), because we have no conjugate points, as the geodesic from \(m_{\varepsilon}\) to \(q(t)\) is strictly minimizing even past \(t=0\) and \(s=D\).
The field \(\partial_s\) is a unit vector field varying smoothly, and the field \(\partial_t\) is a Jacobi vector field along each geodesic, with given values at each end.
The geodesic at \(t=0\) is the one from \(m_{\varepsilon}\) through \(c\); as \(\varepsilon \to 0\) it approaches a limiting geodesic uniformly, the one through \(m_0\) reaching \(c\) with velocity \(w\).
The Jacobi vector field \(\partial_t\) along each geodesic remains a smooth Jacobi vector field as \(\varepsilon \to 0\), with the same initial and final conditions.
Such vector fields are uniformly bounded with all derivatives, by differentiating the Jacobi equation.
By the Arzel\'a--Ascoli theorem, the surface extends to a smooth surface with boundary at \(s=D\) and at \(t=0\) immersed smoothly into \(M\).

For \(s < D\), \(w\) is the gradient of distance from \(m_0\), so \(v\) is tangent to each metric sphere around \(m_0\).
Since there are no conjugate points before the cut locus, the metric sphere around \(m_0\) of radius \(s\) is a smooth hypersurface.
By continuity, the flow lines of \(v\) preserve distance from \(m_0\) at every point of \(M\), so \(v\) is tangent to \(\cut{m_0}\) at every point where our surface hits \(\cut{m_0}\), and so \(w\) is normal.
At \(c\), i.e. when \((s,t)=(D,0)\), \(u=\partial_t\) so \(v=u-(u \cdot w)w\).
In particular, every vector \(u\) is a linear combination \(u=v+a \, w\) of some \(v\) tangent to \(\cut{m_0}\) and some \(w\) normal to \(\cut{m_0}\).
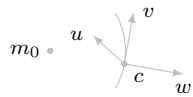
\begin{center}
\begin{tikzpicture}
\draw[\curvecolour] ({cos(30)},{sin(30)}) arc (30:-30:1);
\draw[-latex,\curvecolour] ({cos(-10)},{sin(-10)}) -- ({cos(-10)-.7*sin(-10)},{sin(-10)+.7*cos(-10)}) node[black,right] {\mathlabel{v}};
\draw[\curvecolour,-latex] ({cos(-10)},{sin(-10)}) --  ({.6*cos(20)},{.6*sin(20)}) node[black,left] {\mathlabel{u}};
\draw[\curvecolour,-latex] ({cos(-10)},{sin(-10}) -- ({1.8*cos(-10)},{1.8*sin(-10}) node[black,below] {\mathlabel{w}};
\drawpoint[m_0]{left}{0}{0}
\drawpoint[c]{below right}{cos(-10)}{sin(-10)}
\end{tikzpicture}
\end{center}
By the same reasoning, every tangent vector to \(\cut{m_0}\) is a limit of a sequence of tangent vectors to smooth metric sphere hypersurfaces, and vice versa, so the tangent spaces to \(\cut{m_0}\) are vector spaces.
Vectors in \(T_c M\) which are normal to \(\cut{m_0}\), by our decomposition, are tangent vectors to minimal geodesics from \(m_0\).
\end{proof}

\begin{corollary}
In a Blaschke manifold, all geodesics are periodic and all of the same length, equal to twice the diameter.
\end{corollary}
\begin{proof}
If we follow a geodesic \(p(t)\) until it hits a cut point \(c=p(D)\), then the vector \(p'(D)\) is normal to the cut locus, and so \(-p'(D)\) is also normal to the cut locus, so \(-p'(D)\) is the velocity of a minimal geodesic from \(m_0\).
\end{proof}

If we switch the roles of \(c\) and \(m_0\) above, note that \(m_0 \in \cut{c}\).
Hence for any two points \(p, q \in M\) at distance equal to the diameter of \(M\), the geodesics between those two points have tangents forming a linear subspace in each of the tangent spaces \(T_p M\) and \(T_q M\).

Any compact connected Riemannian manifold has a geodesic loop (perhaps not smoothly periodic) in every homotopy class.
Any Blaschke manifold has all geodesics periodic of the same length, so the periodic geodesics lie in a connected family, i.e. there is at most one nontrivial homotopy class.
Turning the geodesic to point in the opposite direction, we see that this homotopy class is its own negative: \(\pi_1(M)=\Z{}/2\Z{}\) or \(\pi_1(M)=0\).

\section{Cohomology}
Take a Blaschke manifold \(M\) and a point \(m_0 \in M\) and let \(P\) be the space of rectifiable paths in \(M\) issuing from \(m_0\) and \(P_m \subset P\) the paths ending at a point \(m \in M\).
The energy function \(\gamma \mapsto \int \left|\dot\gamma\right|^2\) is a Morse function on \(P_m\), with critical points the geodesics from \(m_0\) to \(m\), so \(P_m\) is homotopy equivalent to a \(CW\)-complex, with cells added at each critical energy level, of dimension given by the index \cite{Reznikov:1985b,Reznikov:1994}.
Each geodesic leaving \(m\) is a union of geodesic loops together with a minimal geodesic ``end'', so that the index of each geodesic is determined by how many times it wraps around.
The end point map \(P \to M\) is a fibration with fibers \(P_m\). The space \(P\) is contractible: suck spaghetti into your mouth.
So the exact sequence in homotopy ensures \(\pi_i\pr{P_{m_0}}=\pi_{i+1}\pr{M}\) and yields a Gysin sequence \(\dots \to \Cohom{i}{P} \to \Cohom{i-k}{M} \to \Cohom{i+1}{M} \to \Cohom{i+1}{P} \to \dots\), where \(k\) is the index of any of the minimal periodic geodesics.
The maps \(\Cohom{i-k}{M} \to \Cohom{i+1}{M}\) are cup product maps with an element \(h \in \Cohom{k}{M}\): \cite{McCleary:2001} p. 143 example 5.C. 
Contractibility of \(P\) makes the exact sequence a collection of isomorphisms, and computes the low degree cohomology of \(M\), showing that it is generated by \(h\), the ``hyperplane class''. 
Duality and connectivity of \(M\) computes the rest of the cohomology groups of \(M\).  
The cohomology ring structure is a consequence purely of  the cohomology group dimensions, using foundational results on cohomology operations \cite{Adams:1960,Adem:1953,Bott:1954,Milnor:1958,Samelson:1963}.

\section{Volume and symplectic volume}\label{section:Volume.and.symplectic}
To find the volume of a Blaschke manifold \(M\) of given diameter, look at the circle bundle \(S^1 \to UTM \to \geod{M}\) from the unit tangent bundle to the space of oriented geodesics.
For any principal circle bundle, or associated complex line bundle, with connection 1-form \(\xi\), on any manifold, the first Chern class is the cohomology class of \(d\xi/2 \pi\).
So if we rescale the metric so that the fibers of \(UTM \to \geod{M}\) have length \(2 \pi\), then the contact form is the connection 1-form, which has curvature the symplectic form on \(\geod{M}\).
So the cohomology class of the symplectic form is the first Chern class, i.e. the Euler class, of the circle bundle.
Fubini's theorem relates the symplectic volume of \(\geod{M}\) to the volume of \(UTM\) and relates that to the volume of \(M\).
For low dimensional Blaschke manifolds and homology spheres, the Gysin sequences of the two fibrations \(M \leftarrow UTM \to \geod{M}\) give the cohomology ring on \(\geod{M}\) \cite{Besse:1978} 2.C, \cite{Yang:1982}.
More generally, Morse theory on the loop space computes the equivariant loop space cohomology which computes the cohomology ring on \(\geod{M}\) and ensures that it matches the model cohomology ring \cite{Reznikov:1985b,Reznikov:1994}, giving the symplectic volume, so \(M\) has the same volume as the model.

\section{Conjugate locus geometry}

Suppose that \(M\) is a Blaschke manifold and \(m_0 \in M\) a point.
Each cut point of \(m_0\) lies on a periodic geodesic made of two minimal geodesics.
By minimality, there are no conjugate points along either of those geodesics.
So a conjugate point doesn't arise except perhaps either at the cut point, or as the geodesic returns to where it started.
As the geodesic returns, it acheives a conjugate point because all of the geodesics simultaneously return to  where they started, all at length equal to twice the diameter. 
So the conjugate points are all at either diameter or twice diameter.

At each cut point of \(m_0\), the normal space to \(\cut{m_0}\) consist precisely of tangents to minimal geodesics to \(m_0\).
If that normal space has dimension 2 or more, i.e. if the cut locus tangent spaces are not hyperplanes, then each unit speed minimal geodesic deforms through a family of unit speed minimal geodesics with fixed end points, i.e. the cut point is a conjugate point.

Conversely, if the normal space has dimension 1, there is no such deformation, so no conjugate point.
Therefore there is no conjugate point on any nearby geodesic.
So all nearby geodesic loops are minimal length loops.
The set of minimal length loops is closed in the space of geodesics.
So all of the geodesics reach a conjugate point just when they loop.
If some geodesic reaches its first conjugate point at twice diameter, then they all do.
Therefore if some geodesic reaches its first conjugate point at diameter, then they all do.

Either way, the exponential map reaches the conjugate locus along a sphere in each tangent space.
Inside that sphere, the exponential map is either a diffeomorphism to the complement of the cut locus, or wraps twice around \(M\).
Since all conjugate points occur at a fixed distance (either \(D\) or \(2D\) from initial point, if \(D\) is the diameter), conjugate points cannot ``scatter'' as we perturb a geodesic, i.e. the index of a conjugate point is constant in any family of geodesics.

Rescale \(M\) so that conjugate points arise along geodesics of length 1.
So \[\exp_{m_0} \colon T_{m_0} M \to M\] has constant rank, say \(k\), along the unit sphere \(S=S^{n-1} \subset T_{m_0} M\).
By constant rank, the exponential map takes \(S\) into an immersed submanifold \(\bar{S} \subset M\) of dimension \(k-1\) (perhaps not embedded), a cover of the conjugate locus.
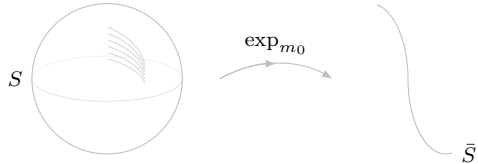
\begin{center}
\begin{tikzpicture}
\draw[\verylightcurvecolour] (1,0) arc[start angle=0, end angle=180, x radius=1cm,y radius=.3cm];
\draw[\lightcurvecolour] (1,0) arc[start angle=0, end angle=-180, x radius=1cm,y radius=.3cm];
\draw[\curvecolour] (0,0) circle (1cm);
\node[left] at (-1,0) {\mathlabel{S}};
\draw[\curvecolour,-latex] (1.5,0) arc (120:60:1.5cm);
\draw[\curvecolour,-latex] (1.5,0) arc (120:90:1.5cm) node[black,above] {\mathlabel{\exp_{m_0}}};
\foreach \i in {45,50,...,70}{
	\draw[\mediumcurvecolour] (.5,{-.5+0.01*\i}) arc[start angle=10, end angle=60, x radius=1,y radius={0.01*\i}];
}
\draw[\curvecolour] (4,0) arc [start angle=0, end angle=80, x radius=.5cm,y radius=1cm];
\draw[\curvecolour] (4,0) arc [start angle=-180, end angle=-80, x radius=.5cm,y radius=1cm] node[black,right] {\mathlabel{\bar{S}}};
\end{tikzpicture}
\end{center}
All of the geodesics from \(m_0\) reach the conjugate locus at the same time and perpendicularly by the equation of first variation.
So time 1 geodesic flow maps \(S\) to the unit normal bundle of \(\bar{S}\).
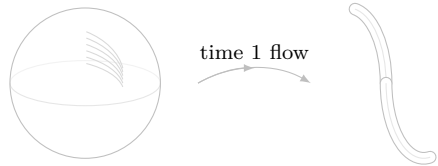
\begin{center}
\begin{tikzpicture}
\draw[\verylightcurvecolour] (1,0) arc[start angle=0, end angle=180, x radius=1cm,y radius=.3cm];
\draw[\lightcurvecolour] (1,0) arc[start angle=0, end angle=-180, x radius=1cm,y radius=.3cm];
\draw[\curvecolour] (0,0) circle (1cm);
\draw[\curvecolour,-latex] (1.5,0) arc (120:60:1.5cm);
\draw[\curvecolour,-latex] (1.5,0) arc (120:90:1.5cm) node[black,above] {\footnotesize{time 1 flow}};
\foreach \i in {45,50,...,70}{
	\draw[\mediumcurvecolour] (.5,{-.5+0.01*\i}) arc[start angle=10, end angle=60, x radius=1,y radius={0.01*\i}];
}
\draw[\curvecolour,double distance=4pt,line cap=round] (4,0) arc [start angle=0, end angle=80, x radius=.5cm,y radius=1cm];
\draw[\lightcurvecolour] (4,0) arc [start angle=0, end angle=80, x radius=.5cm,y radius=1cm];
\draw[\curvecolour,double distance=4pt,line cap=round] (4,0) arc [start angle=-180, end angle=-80, x radius=.5cm,y radius=1cm];
\draw[\lightcurvecolour] (4,0) arc [start angle=-180, end angle=-80, x radius=.5cm,y radius=1cm];
\end{tikzpicture}
\end{center}
The time 1 geodesic flow is a diffeomorphism lying ``above'' the exponential map.
The unit normal bundle of \(\bar{S}\) has the same dimension as \(S\).
By compactness of \(S\), time 1 flow takes \(S\) diffeomorphically to the unit normal bundle of \(\bar{S}\).
Either the conjugate locus is the cut locus or else conjugate points arise when a geodesic reaches a first period.
In either case, the geodesics (minimal or periodic) have tangents consisting of the normal space to the conjugate locus metric sphere by proposition~\vref{proposition:cut}.
The fibers of \(\exp_{m_0} \colon S \to M\) are great spheres.
The map \(S \to \bar{S}\) is a fiber bundle, fibering the unit sphere \(S \subset T_{m_0} M\) in any tangent space of \(M\), with great spheres as fibers.

\begin{theorem}\label{theorem:2.to.1}
Suppose that \(M\) is a Blaschke manifold of dimension \(n\).
The universal covering space \(M' \to M\) is either equal to \(M\) or a 2-1 covering of \(M\), and is a Blaschke manifold in the pullback metric.
The cut locus of any point of \(M'\) is a smooth embedded submanifold equal to the conjugate locus.
For a 2-1 covering, the cut locus and conjugate locus of any point of \(M'\) is a point.
Rescale \(M\) so that \(M'\) has unit diameter.
For any point \(m_0 \in M\), there is smooth fiber bundle mapping \(S \to \bar{S}\) of the unit sphere \(S \subset T_{m_0} M\), each of whose fibers is a great sphere, so that \(M'\) is homeomorphic to the quotient \(\bar{B}/\!\!\sim\) of the closed unit ball \(\bar{B} \subset T_{m_0} M\) by the equivalence \(x \sim y\) if either \(x=y\) or \(x, y \in S\) and \(x\) and \(y\) are mapped to the same point in \(\bar{S}\). 
\end{theorem}
\begin{proof}
Suppose that the cut locus of \(M\) is not the conjugate locus.
As above, the conjugate locus of any point is just that point. 
All geodesics from a point become conjugate just when they return to that point.
The cut locus is half-way along each geodesic, so a hypersurface.
From our study of the cohomology, \(M\) is a cohomology real projective space, so its universal covering \(M'\), i.e. its set of oriented tangent planes, is a 2-1 covering, with a pullback metric with conjugate locus of each point exactly the antipodal point.
The distance between antipodal points on \(M'\) is exactly twice the diameter of \(M\), because a path from one to another projects to a loop of nontrivial homology in \(M\).
On the other hand, one easily arrives at a shorter path for nonantipodal points. 
Hence \(M'\) has diameter double that of \(M\), with cut locus of each point just its antipodal point.

So we can assume that \(M'=M\) has conjugate locus equal to cut locus.
Scale to have unit diameter.
The  exponential map on the closed unit ball in any tangent space of \(M\) identifies \(M=\bar{B}/\sim\) as topological spaces.
\end{proof}

\begin{corollary}
If the great sphere fiber bundles on unit spheres of two Blaschke manifolds are  isomorphic as topological fiber bundles then the Blaschke manifolds are homeomorphic.
\end{corollary}

\begin{corollary}\label{corollary:sphere.and.rp}
Suppose that \(M\) is a Blaschke manifold and that the cut locus of some (hence any) point is a hypersurface.
Then \(M\) is diffeomorphic to real projective space and its cut locus diffeomorphic to a linear real projective hypersurface.
\end{corollary}
\begin{proof}
The cut point occurs halfway along each periodic geodesic; the sphere bundle identifies opposite points, and only those because the fiber of the sphere fibration has dimension zero.
There is no conjugate point near any cut point, so the exponential map is a local diffeomorphism near each cut point.
In particular, the exponential map is a surjective local diffeomorphism \(\bar{B} \to M\) from the closed ball \(\bar{B}\) of radius equal to the diameter of \(M\).
The smooth functions on \(M\) are identified with the smooth functions on that ball invariant under the antipodal map.
Any two such Blaschke manifolds \(M\) have the same smooth functions and so are diffeomorphic.
Note that the cut locus is also diffeomorphic to real projective space.
\end{proof}

\section{Cut locus topology}

Let \(m^{\perp}\) be the cut locus of each point \(m \in M\).
As above, \(M-m^{\perp}\) is diffeomorphic to an open ball, so its cohomology is trivial.
For two distinct points \(m_0, m_1 \in M\), if the intersection \(m_0^{\perp} \cap m_1^{\perp}\) is empty, then \(m_0^{\perp} \subset M-m_1^{\perp}\) lies in a contractible open subset of \(M\), so the inclusion \(m_0^{\perp} \to M\) gives a trivial morphism in cohomology \(\Cohom{*}{M} \to \Cohom{*}{m_0^{\perp}}\).
As a CW complex, \(M\) is a ball glued to \(m_0^{\perp}\), so all cohomology of \(M\) except in top dimension injects via \(\Cohom{*}{M} \to \Cohom{*}{m_0^{\perp}}\), making \(M\) a cohomology sphere.
So if \(M\) is a Blaschke manifold not modelled on a sphere, then any two points have intersecting cut loci.
Take a point \(m \in m_0^{\perp} \cap m_1^{\perp}\).
The normal space to \(m_0^{\perp}\) at \(m\) is the collection of tangents of geodesics from \(m_0\) to \(m\).
So the normal spaces to \(m_0^{\perp}\) and \(m_1^{\perp}\) at \(m\) are disjoint linear subspaces.
Therefore the intersection \(m_0^{\perp} \cap m_1^{\perp}\) is a smooth submanifold.
It is not known if the intersection of 3 cut loci is a smooth submanifold.
At a point \(m \in m_0^{\perp}-m_1^{\perp}\), the unique geodesic from \(m\) to \(m_1\) is \emph{not} normal to \(m_0^{\perp}\), so has nonzero projection to the tangent space to \(m_0^{\perp}\).
So if \(m_1 \in m_0^{\perp}\) then the function \(m \mapsto d\of{m_1,m}^2\) on \(m_0^{\perp}-m_1^{\perp}\) has a unique critical point at \(m_1\), and its reverse gradient flow identifies \(m_0^{\perp}-m_1^{\perp}\) with a ball.
So as a CW complex, \(m_1^{\perp}\) is a ball glued to some lower dimensional manifold.
The dual of the hyperplane class, having the same dimension as the cut locus of any point, is a multiple of the homology class of that cut locus, and vice versa.
As a CW complex, \(\cut{m_0}\) only has one cell of the dimension of the cut locus, the open cell, so the homology classes agree (up to sign, but the hyperplane class was only defined up to a sign).

\section{Quotienting by an involution}

\begin{proposition}\label{proposition:involution}
Suppose that \(M\) is a Blaschke manifold and each point of \(M\) has discrete cut locus.
Then each point \(m \in M\) has cut locus a point \(s(m)\) and the map \(s \colon M \to M\) is an isometry with \(s ^{-1} = s\).
The quotient \(\bar{M}=M/s\) is also a Blaschke manifold, diffeomorphic to real projective space, with cut locus a hypersurface.

Conversely, if \(\bar{M}\) is a Blaschke manifold diffeomorphic to real projective space, then the universal covering space \(M\) of \(\bar{M}\) is a Blaschke manifold diffeomorphic to a sphere and each point of \(M\) has cut locus a point. 
\end{proposition}
\begin{proof}
Suppose \(M\) has discrete cut locus, hence a single point and \(M\) is homeomorphic to a sphere.
Sliding a point \(m\) along a geodesic, we slide its antipodal point \(s(m)\).
Take two points \(m\) and \(n\), not conjugate, and take the minimal geodesic between them, and then slide along it until we get \(m\) and \(n\) to slide into \(s(m)\) and \(s(n)\) simultaneously, with the minimal geodesic sliding into place between them, so \(s\) preserves distance.
Therefore \(s\) is smooth and the quotient \(\bar{M}\) is a compact connected Riemannian manifold.
There  are no conjugate points along a minimal geodesic on \(\bar{M}\) because a conjugate point lifts up to \(M\).
Imagine that there are two distinct minimal geodesics between two points of \(\bar{M}\).
Upstairs on \(M\), each point has two preimages, and a unique minimal geodesic between each of the pairs, which must map to these two minimal geodesics on \(\bar{M}\); these two minimal geodesics must therefore fit into one periodic geodesic.
Therefore there is a unique minimal geodesic connecting any two points of \(\bar{M}\) unless the points lie at a distance equal to diameter, in which case there are two minimal geodesics.
So \(\bar{M}\) is Blaschke and has index zero, i.e. a hypersurface as cut locus.

On the other hand, if \(\bar{M}\) is a Blaschke manifold with cut locus a hypersurface, our cohomology calculation says that \(\bar{M}\) has the homology and homotopy groups of real projective space, and the index of our minimal geodesic at each cut point is zero, so the sphere fibration is that of real projective space, i.e. \(\bar{M}\) is diffeomorphic to real projective space.
Theorem~\vref{theorem:2.to.1} says that there is a finite covering Blaschke manifold \(M\) with cut locus equal to conjugate locus, so not isometric to \(\bar{M}\).
So the covering is 2-1, \(M\) has discrete cut loci, and is diffeomorphic to a Euclidean sphere.
\end{proof}

\section{Volume and Jacobi vector fields}\label{section:Volume.and.Jacobi}

Suppose that \(M\) is a Blaschke manifold of diameter \(D\).
The exponential map takes a ball \(B=B_D(0) \subset T_p M\) onto a dense open ball in \(M\), so \(\vol{M}=\int_B \det \exp'\).
But \(\det \exp'(v)\) is the wedge product of a basis of Jacobi vector fields, orthonormal at the origin, along the geodesic with velocity \(v\).
The Jacobi vector field equation contains the curvature.
Berger (\cite{Besse:1978} appendix D) proved that the volume of a Blaschke manifold exceeds that of the sphere of the same diameter unless the Blaschke manifold is isometric to the sphere \cite{Besse:1978} appendix E.  
More generally, the injectivity radius \(\inj{M}\) of any compact connected Riemannian manifold \(M\) satisfies
\begin{equation}\label{equation:Berger.inequality}
\frac{\inj[n]{M}}{\vol{M}} \le \frac{\inj[n]{S^n}}{\vol{S^n}}
\end{equation}
where \(S^n\) is an \(n\)-dimensional round sphere in Euclidean space, with equality just when \(M\) is isometric to such a sphere; see \cite{Chavel:2006} pp. 319--331 for a complete proof. 
The proof is a subtle analysis of the Jacobi vector field equation along each geodesic, making use of periodicity and reversibility of geodesics, and a complicated integral inequality proved by Kazdan \cite{Kazdan:1982}.
Combined with the results on volume above, this proves the Blaschke conjecture for homology spheres.
The Blaschke conjecture for homology real projective spaces follows by taking a double cover.

\section{Green's proof for surfaces}
To give a taste for Berger's arguments for Blaschke manifolds modelled on spheres, we give Green's proof of the Blaschke conjecture for surfaces.

\begin{lemma}\label{lemma:Berger.inequality}
Along any minimal geodesic of length \(\ell\) on any surface of Gauss curvature \(K\),
\[
\int_0^{\ell} K(s) \sin^2\of{\frac{\pi s}{\ell}} \le \frac{\pi^2}{2\ell},
\]
with equality just when
\[
K(s) = \pr{\frac{\pi}{\ell}}^2
\]
is the same constant curvature as along a geodesic on a Euclidean sphere of Riemannian diameter \(\ell\).
\end{lemma}
\begin{proof}
Take geodesic normal coordinates
\(
ds^2 = dr^2 + h(r,\theta)^2 \ d\theta^2
\).
The Sturm--Liouville operator
\[
-\frac{d^2}{ds^2} - K(s)
\]
along a geodesic is self-adjoint on square integrable functions vanishing on endpoints.
The function \(h\) is null for this operator and vanishes at the origin, and vanishes next just when we hit a conjugate point.
By the Sturm comparison theorem, the eigenvalues of this operator are positive or zero.
Plug the function \(\sin\of{\pi s/\ell}\) into the operator.
\end{proof}

\begin{lemma}
Suppose that \(M_0\) is a 2-sphere or real projective plane, with standard round metric.
The injectivity radius \(\inj{M}\) of any Riemannian metric on a compact surface \(M\) diffeomorphic to \(M_0\) satisfies
\[
\frac{\inj[2]{M}}{\area{M}} \le \frac{\inj[2]{M_0}}{\area{M_0}} 
\]
with equality just when \(M\) is isometric to \(M_0\) up to constant positive rescaling.
\end{lemma}
\begin{proof}
Let \(\ell=\inj{M}\), \(A=\area{M}\) and \(\chi\) be the Euler characteristic of \(M\).
Consider the bundle \(B \to M\) consisting of tuples \((p,q,u,s)\) so that \(p, q \in M\), \(0 \le s \le \ell\) and \(u \in T_p M\) is a unit vector with \(q=\exp_p (su)\).
Clearly \(p \colon (p,q,u,s) \in B \mapsto p \in M\) is a fiber bundle, with fibers \(UT_p M \times [0,\ell]\), so \(B\) is a compact 4-manifold with boundary.
Note that there is another bundle map: \(q \colon (p,q,u,s) \in B \mapsto q \in M\).
Let \(\xi\) be the connection 1-form on \(UTM\), which we pullback to \(B\) by a third bundle map \(u \colon (p,q,u,s) \in B \mapsto u \in UTM\).
Recall that along any curve of the form \(u\defeq e^{i\theta}u_0\) in \(UTM\), \(\omega=d\theta\).
On \(B\), take the differential form
\[
\Omega \defeq \pr{\frac{\pi^2}{2 \ell^2} p^* dA - q^*\pr{K dA} \sin^2\of{\frac{\pi s}{\ell}}} \wedge ds \wedge \xi.
\]
Pushing down (integrating over the fibers),
\(u_* \Omega = f \, dA \wedge \xi\)
where 
\[
f(u) \defeq \int_0^{\ell} \pr{\frac{\pi^2}{2 \ell^2} - K\of{\exp_p\of{su}} \sin^2\of{\frac{\pi s}{\ell}}} \, ds.
\]
Lemma~\vref{lemma:Berger.inequality} says  that \(f \ge 0\) so \(\int_B \Omega=\int_{UTM} u_* \Omega \ge 0\).
Moreover, \(\int_B \Omega=0\) just when the curvature is constant and equal to the curvature of the model \(M_0\).
On the other hand, 
\begin{align*}
\int_B q^*\pr{K dA} \wedge \sin^2\of{\frac{\pi s}{\ell}} ds \wedge \xi
&=
2 \pi \int_M K \, dA \int_0^{\ell} \sin^2\of{\frac{\pi s}{\ell}} \, ds,
\\
&=
\pi  \chi \ell,
\end{align*}
while
\begin{align*}
\int_B 
\frac{\pi^2}{2 \ell^2}
p^*dA \wedge ds \wedge \xi
&=
\frac{\pi^3 A}{\ell},
\end{align*}
so that 
\[
0 \le 
\int_B \Omega = \pi^4 \ell
\pr{
\frac{A}{\ell^2}-\frac{2 \chi}{\pi}
}
\]
with equality just for \(M\) of constant curvature equal to that of \(M_0\), so isometric to \(M_0\).
\end{proof}

The same trick gives some information in all dimensions:
\begin{theorem}\label{theorem:scalar}
Take a compact connected Riemannian manifold \(M\).
Let \(\delta_M\) be the \emph{conjugate radius}: the shortest distance along any geodesic between a point and its first conjugate point.
Denote the scalar curvature of \(M\) as \(\scalar{M}\).
Let \(M_0\) be a positively curved space form (i.e. compact connected Riemannian manifold of constant sectional curvature) of the same dimension as \(M\), for example a Euclidean sphere, of any radius, or a real projective space.
Then
\[
\frac{\delta^2_M \int_M \scalar{M}}{\vol{M}} 
\le 
\frac{\delta^2_{M_0} \int_{M_0} \scalar{M_0}}{\vol{M_0}},
\]
with equality just when \(M\) is also a positively curved space form,
with constant positive scalar curvature equal to \(\pi^2/\delta_M^2\).
\end{theorem}
Note that both sides of the inequality are invariant under rescaling the metrics by positive constants.
\begin{proof}
Rescale to arrange \(\delta_M=\delta_{M_0}=\pi\) and let \(n\) be the dimension of \(M\).
Take a unit speed geodesic \(\gamma \colon [0, \pi] \to M\).
Let \(\nu = \nu(s)\) be a parallel unit normal vector field along \(\gamma\). 
The geodesic \(\gamma\) has no conjugate points except perhaps its end points.
The vector field \(X=\sin(s)\nu\) gives nonnegative second variation of energy as we vary the geodesic \cite{Chavel:2006} p. 82 theorem II.5.4:
\[
0 \le \int_0^{\pi} \pr{\left|\nabla_s X\right|^2 - \left<R(X,\dot\gamma)\dot\gamma,X\right>} \, ds
\]
with equality just for Jacobi vector fields.
Applied to our vector field \(X=\sin(s)\nu\), 
\[
\int_0^{\pi} \sin^2(s) \left<R(\nu,\dot\gamma)\dot\gamma,\nu\right> \, ds \le \frac{\pi}{2}
\]
with equality just when \(X\) is a Jacobi vector field, that is, just when the sectional curvature is \(K(\nu,\dot\gamma)=1\).
Integrate this inequality over all choices of parallel unit normal vector fields \(\nu\), each given an arbitrary choice of initial vector \(\nu(\gamma(0))\):
\[
\int_0^{\pi} \sin^2(t) \Ricci{\dot\gamma} \, dt \le (n-1)\frac{\pi}{2}.
\]
Equality occurs just when any 2-plane containing the tangent line to \(\gamma\) has unit sectional curvature.
For any tangent vector \(v\), let \(g^t v\) be the geodesic flow of \(v\) at time \(t\), i.e. the tangent vector at time \(t\) to the constant speed geodesic with initial velocity \(v\).
The Liouville measure \(dL\) is \(g^t\)-invariant.
Integrate over the unit tangent bundle \(UTM\), integrating over all points of all unit speed geodesics:
\begin{align*}
\int_{UTM} \int_0^{\pi} \sin^2(t) \Ricci{g^t v} \, dt \, dL
&=
\int_0^{\pi} \sin^2(t) \int_{UTM} \Ricci{g^t v} \, dL \, dt,
\\
&=
\int_0^{\pi} \sin^2(t) \int_{UTM} \Ricci{v} \, dL \, dt,
\\
&=
\int_0^{\pi} \sin^2(t) dt \int_{UTM} \Ricci{v} \, dL,
\\
&=
\frac{\pi}{2} \int_{UTM} \Ricci{v} \, dL,
\\
&=
\frac{\pi}{2} \int_{m \in M} \int_{v \in UT_m M} \Ricci{v} \, dv \, dm,
\\
&=
\frac{\pi}{2n} \vol{S^{n-1}} \int_M \scalar{} \, dm.
\end{align*}
\end{proof}

\section{Sphere fibrations of spheres}\label{section:sphere.fibrations}

As we have seen, the homeomorphism type of any Blaschke manifold \(M\) is encoded in a fibration of a Euclidean sphere \(UT_m M \subset T_m M\) by great spheres.
For example, if \(M=\CP{n}\) then the unit sphere \(S^{2n-1} \subset T_m M=\C{n}\) is fibered over the cut locus of \(m\) by the Hopf fibration: \(S^1 \to S^{2n-1} \to \CP{n-1}\).
Each row in the table on page~\pageref{table:know} (except the first two) arose by proving that any great sphere fibration of a sphere is isomorphic to a Hopf fibration, with various notions of isomorphism.
The generic great sphere fibration does not arise from a Blaschke metric.
Great sphere fibrations are of independent interest in the theory of nonlinear elliptic systems of partial differential equations \cite{McKay:2004}.
The homotopy theory of great sphere fibrations remains a basic question, related both to the Blaschke conjecture and to the homotopy theory of elliptic systems, and thus to continuity methods for solving elliptic systems.

Take a great sphere fibration \(S^{k-1} \to S^{n+k-1} \to X^n\). 
Identify each great sphere with the linear subspace it spans in \(\R{n+k}\): the sphere fibration embeds the base \(X\) into a Grassmannian \(X \subset \Gr{k}{n+k}\).
So great sphere fibrations are identified with certain submanifolds of Grassmannians.

For each \(k\)-plane \(P \subset \R{n+k}\), the \emph{bad set} of \(P\), denoted \(B_P \subset \Gr{k}{n+k}\) is the set of all \(k\)-planes with nontrivial intersection with \(P\); call \(P\) the \emph{vertex} of its bad set: \cite{Gluck/Warner/Yang:1983} p. 1047.
Split \(\R{n+k}=P \oplus P^{\perp}\), say, and write every linear subspace \(Q\) close enough to \(P\) as the graph of a linear map in \(P^* \otimes P^{\perp}\).
The bad set is the set of linear maps in \(P^* \otimes P^{\perp}\) which are not injective, so the bad set is a cone, at least near \(P\).
Take a great sphere fibration \(S^{k-1} \to S^{n+k-1} \to X^n\) with associated image \(X^n \subset \Gr{k}{n+k}\).
Because \(X\) is a great sphere fibration, each point of \(S^{n+k-1}\) can only lie in one fiber of \(X\), i.e. each unit vector \(v \in \R{n+k}\) which lies in a plane \(P \in X\) can not lie in any other plane \(Q \in X\), so \(Q \notin B_P\).
In other words, \(X\) intersects each bad set only at its vertex.

A \emph{bad cone} is the tangent cone of a bad set at its vertex.
The fiber bundle mapping \(f \in \operatorname{Epi}\of{\R{n+k},\R{k}} \mapsto \ker(f) \in \Gr{k}{n+k}\) has a local section, say \(P \mapsto s(P)\), unique up to replacing \(s(P)\) by \(g(P)s(P)\) for a map \(g\) valued in \(\GL{k,\R{}}\).
For each \(v \in T_P \Gr{k}{n+k}\), the linear map \(s(P)^{-1} \left.s'(P)v\right|P \colon P \to \R{n+k}/P\) is invariantly defined, giving a linear isomorphism \(T_P \Gr{k}{n+k}=P^* \otimes \pr{\R{n+k}/P}\).
This isomorphism identifies the bad cone with the set of all noninjective linear maps \(P^* \otimes \pr{\R{n+k}/P}\).

A compact connected submanifold \(X^n \subset \Gr{k}{n+k}\) arises from a great sphere fibration just when  \(X\) intersects each bad set either nowhere or only at its vertex transversely \cite{McKay:2004}.
Transversality with bad sets is an open condition on a submanifold of the Grassmannian, so if \(X \subset \Gr{k}{n+k}\) is a submanifold arising from a great sphere fibration, then every \(C^1\)-small perturbation of \(X\) also arises from a unique great sphere fibration.
We can express transversality with bad sets in a simple intrinsic description.
Once again every tangent vector to a Grassmannian is intrinsically identified with a linear map: \(T_P \Gr{k}{n+k}=P^* \otimes \left(\R{n+k}/P\right)\).
Take a submanifold \(X \subset \Gr{k}{n+k}\).
Each tangent space \(T\defeq T_P X\) is a linear subspace \(T \subset T_P \Gr{k}{n+k}\), so a linear space of linear maps, \(T \subset P^* \otimes \left(\R{n+k}/P\right)\) or alternatively is a bilinear map \(T^* \otimes P^* \otimes \left(\R{n+k}/P\right)\).
Each vector \(v \in P\) is a linear map \(T \to \R{n+k}/P\).
A submanifold \(X \subset \Gr{k}{n+k}\) intersects a bad cone \(B_P\) at its vertex transversely just when this linear map is a linear isomorphism for every \(v \ne 0\): \cite{Gluck/Warner/Yang:1983} p. 1047.

The simplest choice of dimensions:
\begin{lemma}%
[\cite{Gluck/Warner/Yang:1983} p. 1046]%
\label{lemma:projective.plane.hopf}
Any great sphere fibration \(S^{k-1} \to S^{2k-1} \to X^k\) has base \(X\) homeomorphic to \(S^k\).
\end{lemma}
\begin{proof}
Take a fiber of \(X\), a great \((k-1)\)-sphere \(S \subset S^{2k-1}\).
Pick any great \(k\)-sphere \(S^+\) containing \(S\), so \(S \subset S^+\) divides \(S^+\) into two hemispheres; choose one.
Every point of \(X-S\) represents a great \((k-1)\)-sphere, and intersects \(S^+\) in two antipodal points, one in our chosen hemisphere, a bijection of \(X-S\) with that hemisphere.
Identify \(X\) with the one-point compactification of that hemisphere.
\end{proof}

\begin{lemma}\label{lemma:projective.plane.bundle}
Any two great sphere fibrations \(S^{k-1} \to S^{2k-1} \to X^k\) and \(S^{k-1} \to S^{2k-1} \to X_0^k\) admit a topological bundle isomorphism
\begin{center}
\begin{tikzcd}
S^{2k-1} \arrow{r} \arrow{d} & S^{2k-1} \arrow{d} \\
X^k \arrow{r} & X_0^k.
\end{tikzcd}
\end{center}
\end{lemma}
\begin{proof}
Take any distinct points \(P_0, P_1 \in X\), corresponding to some linear subspaces.
With elementary linear algebra, arrange that  \(P_0=\R{k} \oplus 0, P_1 = 0 \oplus \R{k} \subset \R{2k}\) and \(\Set{y=x} \in X\), and the same for \(X_0\), so that \(X\) intersects \(X_0\) at the points \(P_0, P_1\) and \(\Set{y=x}\).
Pick a unit vector \(v \in P_1\).

For any vector \(x_0 \in \R{k}-0\), if \(P=\Set{y=Ax}\), let \(\ev[x_0]{P}{}\defeq Ax_0 \in \R{k}\).
No two \(k\)-planes in \(X\) intersect away from the origin, so \(\ev{x_0}{} \colon X-P_1 \to \R{k}\) is 1-1.
Every vector in \(\R{2k}\) lies in one of the \(k\)-planes in \(X\), so \(\ev{x_0}{}\) is onto.
Clearly
\[
\ev[x_0]{P}{'}A_0 = A_0 x_0
\]
is 1-1, so \(\ev{x_0}{}\) is a diffeomorphism.
The equation \(\ev[x_0,X]{P}{} = \ev[x_0,X_0]{Q}{}\) has a unique solution \(Q=\phi(P)\), a diffeomorphism \(\phi \colon X-P_1 \to X_0-P_1\), extending to a bijection \(\phi \colon X \to X_0\) by \(\phi\of{P_1}\defeq P_1\).
In other words \(\phi\Set{y=Ax}=\Set{y=Bx}\) just when \(Ax_0=Bx_0\).

We want to show that \(\phi\) is continuous at \(P_1\).
Clearly \(\phi^{-1}\) is given by the same construction, swapping \(X\) and \(X_0\); if we prove that \(\phi\) is continuous then \(\phi\) is a homeomorphism.
Take some \(P(t) \to P_1\) in \(X\), say \(P(t)=\Set{y=A(t)x}\).
We can assume that \(A(t)^{-1}=t A_0 + O(t)^2\), with \(A_0 \ne 0\).
Because \(X\) is not tangent to the bad cone, \(A_0\) is also invertible.
Write \(Q(t)=\phi(P(t))=\Set{y=B(t)x}\).
We have
\(
A(t) x_0 = B(t) x_0,
\)
so that 
\[
B(t)^{-1}\pr{t A_0 + O(t)^2}^{-1} x_0 = x_0, 
\]
so that
\[
B(t)^{-1}A_0^{-1} x_0 = O(t).
\]
Again by transversality to the bad cone, the map
\[
\Set{x=Cy}\in X \mapsto CA_0^{-1} x_0 \in \R{k}
\] 
is a diffeomorphism, so \(B(t)^{-1}=O(t)\).
By transversality to the bad cone, we can solve the linear equation \(B_0 A_0^{-1} x_0=x_0\) uniquely for \(B_0 \in T_{P_1} X_0\).
Expand
\begin{align*}
\pr{B(t)^{-1}-tB_0}A_0^{-1} x_0 
&=
B(t)^{-1}A_0^{-1} x_0 - tx_0,
\\
&=
B(t)^{-1}\pr{tA^{-1} + O(t)}x_0-tx_0,
\\
&=
O(t)^2
\end{align*}
so that \(B(t)^{-1}=tB_0+O(t)^2\).
Therefore \(\phi\) is a homeomorphism.
We leave the reader to check that \(\phi\) is a diffeomorphism just when \(X\) and \(X_0\) are tangent at \(P_1\).

Take any smooth function \(g \colon X \to \R{}\) so that \(0 \le g \le 1\) and \(g\of{P_0}=0\) and \(g\of{P_1}=1\) with nondegenerate minimum and maximum, i.e. \(g''\of{P_0} > 0\) and \(g''\of{P_1} < 0\).
At each \(P=\Set{y=Ax} \in X-\Set{P_1}\), let \(E=(1-g)I+gA^{-1}\) and \(F=AE\).
Since \(A\) has no negative eigenvalues, \(E\) and \(F\) are invertible linear maps.
The standard Hopf fibration \(X_0\) is the set of real, complex, quaternionic or octave lines, determined by some algebra \(\mathbb{K}=\R{}, \C{}, \Ha{}\) or \(\Oc{}\) of linear maps \(J \colon \R{2k} \to \R{2k}\).
For each \(J \in \mathbb{K}\) and point \(P \in X\), let
\[
J_P \colon \pr{x,y} \mapsto \pr{EJE^{-1}x, FJF^{-1} y}. 
\]
For each \(P \in X\), these linear maps \(J_P\) form an algebra isomorphic to \(\mathbb{K}\), and \(P\) is a \(J_P\)-invariant linear subspace.
Check (again using transversality with the bad cone) that \(J_P\) extends continuously to all of \(X\), giving a real, complex, quaternionic or octave linear structure to each subspace \(P \in X\), agreeing with \(J\) on \(P_0\) and \(P_1\).

Take the tautogical vector bundle \(U_X \to X\) with fiber \(U_{X,P}=P\).
Define a bundle isomorphism \(\Phi \colon U_X \to U_{X_0}\) above \(\phi\) on the tautological bundles as follows.
For each point \(P \in X-\Set{P_1}\), say \(P=\Set{y=Ax}\) and with \(Q=\phi(P)=\Set{y=Bx}\), take the vector \(v=\pr{x_0,Ax_0} \in P\) and let \(\Phi\of{v}=v\).
Extend \(\Phi\) by real, complex, quaternionic or octave linearity by requiring \(\Phi \circ J_P=J \circ \Phi\) for \(J \in \mathbb{K}\).
This defines \(\Phi\) away from \(P_1\); on \(v \in P_1\), let \(\Phi(v)=v\).
Since \(S^{2k-1}\) is the set of unit vectors in \(U_X\), \(\Phi \colon S^{2k-1} \to S^{2k-1}\).
\end{proof}

\begin{corollary}
Any Blaschke manifold modelled on a real, complex, quaternionic or octave projective plane is homeomorphic to its model.
\end{corollary}

For a great sphere fibration \(S^3 \to S^7 \to X^4\), the homotopy class of the associated inclusion \(X^4 \subset \Gr{4}{\R{8}}\) is the same as that of the associated inclusion arising from the Hopf fibration \cite{Sato/Mizutani:1984}, so that the Pontryagin number is equal to that arising in the Hopf fibration.
Heavy differential topology ensures \(PL\)-isomorphism with the Hopf fibration.

For a great sphere fibration \(S^1 \to S^{2n+1} \to X^{2n}\), the embedded submanifold \(X \subset \Gr{2}{\R{2n}}\) has tangent spaces \(T \subset P^* \otimes \pr{\R{2n+2}/P}\).
We can always lift \(X\) to the oriented Grassmannian, so assume that each \(P \in X\) is an oriented 2-plane.
Using a little linear algebra, there are canonically determined complex structures on \(T\) and \(\R{2n+2}\) so that \(P \subset \R{2n+2}\) is a complex linear subspace and the inclusion \(T \subset P^* \otimes \pr{\R{2n+2}/P}\) is has image in the complex linear tensor product and is complex linear \cite{McKay:2004}.
So at each point \(P \in X\), we have a chosen complex structure \(J_P \colon \R{2n+2} \to \R{2n+2}\).
Map \(X \to \CP{2n-1}\) by \(P \mapsto \left[x-\sqrt{-1}J_P x\right]\) for \(P=\left[x \wedge J_P x\right]\).
Let \(Z_0\) be a generically chosen linear subspace \(Z_0=\CP{n} \subset \CP{2n-1}\).
If every linear \(\CP{n+1} \subset \CP{2n-1}\) containing \(Z_0\) happens to intersect the image of \(X\) transversally (as happens when \(X\) arises from the Hopf fibration), then it does so at a unique point, so that \(X\) is identified with the set of all linear subspaces \(\CP{n+1} \subset \CP{2n-1}\) containing \(Z_0\), a space diffeomorphic to \(\CP{n}\) \cite{McKay:2004}, and it is easy to extend this diffeomorphism to identify the fibration \(S^{2n+1} \to X^{2n}\) with the Hopf fibration.
(Yang \cite{Yang:1990} correctly criticized the argument of Sato \cite{Sato:1982} but the map \(X \to \CP{2n-1}\) is essentially Sato's.
I incorrectly stated that, for any compact elliptic submanifold \(X\) of the Grassmannian, the generic \(Z_0\) would actually have the property above: transverse intersection of and \(X\) with every linear \(\CP{n+1} \subset \CP{2n-1}\) containing \(Z_0\).)

Any great sphere fibration \(S^3 \to S^{4n+3} \to X^{4n}\) admits a reduction of structure group to \(\operatorname{Sp}(1)\), by examining maps between classifying spaces \cite{Sato:1985}.
The classifying map for this bundle factors through a map \(X \to \HP{n}\), which consequently is an isomorphism on homotopy groups, so a homotopy equivalence by Whitehead's theorem \cite{Hatcher:2002} p. 346 theorem 4.5.
Sato states without proof that, from the \(K\)-theory of \(\HP{n}\), this homotopy equivalence should ensure the existence of a homeomorphism \(X \cong \HP{n}\).

\newcommand{\KK}{\ensuremath{\mathbb{K}}}
Take two great sphere fibrations \(S^{k-1} \to S^{p+k-1} \to X^p\) and \(S^{k-1} \to S^{q+k-1} \to Y^q\), with the same fiber \(S^{k-1}\).
Our standard examples of such are Hopf fibrations arising from vector spaces over \(\KK=\R{}, \C{}\) or \(\Ha{}\) for \(k=1,2,4\).
We can take the sum of two vector spaces, and thereby generate a new Hopf fibration.
By analogy, we can try to create a ``sum'' \(X \oplus Y\).
Suppose that on each subspace \(x \in X \subset \Gr{k}{\R{p+k}}\)
we can pick the structure of a \(\KK\)-vector space, varying smoothly with \(x\).
Each nonzero vector \(v \in \R{p+k}\) lies in a unique such subspace: rescale \(v\) to unit length and apply the fiber bundle map \(S^{p+k-1} \to X^p\).
Hence we have a unique action of the scalars of \(\KK\) on each nonzero vector in \(\R{p+k}\).
Extend to have every scalar in \(\KK\) fix \(0 \in \R{p+k-1}\), so that every element of \(\KK\) acts on \(\R{p+k}\) continuously, and smoothly away from the origin.
Do the same with \(\R{q+k}\).
Allow each scalar in \(\KK\) to act on \(\R{p+k} \times \R{q+k}\) by acting on the factors individually.
The ``sum'' of the great sphere fibrations is the fibration of \(S^{p+q+2k-1}\) whose fibers are the unit vectors in the subspaces invariant under the \(\KK\)-action.

There is an obvious criticism: this approach depends on choice of the two \(\KK\)-actions, one on \(\R{p+k}\) and one on \(\R{q+k}\).
It turns out that for \(\KK=\R{}\) or \(\C{}\), there is a natural choice of such \(\KK\)-action \cite{McKay:2004}.
For \(\KK=\Ha{}\) or \(\KK=\Oc{}\) no natural choice is known, nor is it known if there is such an action.

Consider the smoothness of the \(\KK\)-action near the points \(v=(x,y) \in S^{p+q+2k-1}\) where one of \(x\) or \(y\) vanishes.
The action turns out to be smooth just when one of \(X\) or \(Y\) is precisely a Hopf fibration.
For this reason, my proof \cite{McKay:2004} theorem 1 is invalid.
Yang's papers \cite{Yang:1990,Yang:1993} are similarly flawed, as first noticed by Werner Ballmann and Karsten Grove: the action he calls \(\rho'\) in \cite{Yang:1990} p. 523, which he claims is smooth, is smooth precisely when the great circle fibration (from which he constructs \(\rho'\)) is a Hopf fibration.

\section{The Radon transform}

If \(M\) is a Blaschke manifold of diameter \(D\) with metric \(d\), let 
\[
F=\Set{(p,q)|p, q \in M, d(p,q)=D}.
\]
We have obvious surjective maps
\[
\begin{tikzcd}
{} & F^{2n-k} \arrow{dr}{p} \arrow{ld}[swap]{q} & {} \\
M^n & {} & M^n
\end{tikzcd}
\]
\(p(p,q)=p\), \(q(p,q)=q\).
The \emph{Radon transform} of a volume form \(\Omega\) on \(M\) is \(\Omega'=q_* p^* \Omega\), where \(q_*\) is the pushforward, i.e. integration over the fibers.
Since \(d\) commutes with pullback and pushforward, \(0=d\Omega'\).
Reznikov \cite{Reznikov:1985b} states without proof that \(\Omega'>0\) on the cut locus of any point, i.e. the pullback of \(\Omega'\) to any cut locus is nowhere vanishing with positive integral; for elementary proofs too long to include here see \cite{McKay:2005} p. 9 lemma 9.

\begin{theorem}%
[\cite{McKay:2005}]%
Every Blaschke manifold modelled on \(\CP{2}\) is diffeomorphic to \(\CP{2}\).
\end{theorem}
\begin{proof}
The Radon transform of any volume form is a symplectic form, so our Blaschke manifold \(M\) is a symplectic 4-manifold containing a symplectic 2-sphere (any cut locus).
Using Seiberg--Witten theory, one sees that \(M\) is symplectomorphic to \(\CP{2}\), up to a nonzero constant factor in the symplectic form \cite{Lalonde/McDuff:1996}.
\end{proof}

\section{Smooth projective planes}

An \emph{incidence plane} is a pair of sets \(P, L\) and a set \(F \subset P \times L\).
Elements of \(P, L, F\) are respectively called \emph{points}, \emph{lines} and \emph{pointed lines}.
A point \(p \in P\) is \emph{on} a line \(\ell \in L\) if \((p,l) \in F\). 
An incidence plane is a \emph{projective plane} if
\begin{enumerate}
\item any two distinct points \(p, q\) lie on a unique line \(pq\),
\item any two distinct lines \(\ell, m\) have a unique point \(\ell m\) lying on both of them,
\item there are 4 points, no 3 on the same line.
\end{enumerate}
Swapping \(P\) with \(L\) gives the \emph{dual} projective plane.
A \emph{smooth} projective plane is one where \(P\) and \(L\) are compact connected manifolds of positive dimension and the maps \(p, q \mapsto pq\) and \(\ell,m \mapsto \ell m\) are smooth.
The \emph{standard} projective planes are \(P=\RP{2}, \CP{2}, \HP{2}\) and \(\OP{2}\) with obvious \(L\) and \(F\). 
The space of smooth projective planes modulo diffeomorphism is infinite dimensional, as any \(C^2\) small perturbation of \(F \subset P \times L\) is also a smooth projective plane.
Using heavy differential topology:
\begin{theorem}%
[\cite{Kramer/Stolz:2007,McKay:2005}]%
\label{theorem:proj.plane.diffeo}
The space of points of any smooth projective plane is diffeomorphic to a standard projective plane.
\end{theorem}

\begin{lemma}%
[\cite{Gluck/Warner/Yang:1983,Reznikov:1985,Salzmann:1995}]
If a Blaschke manifold is modelled on a standard projective plane, then it is a smooth projective plane.
\end{lemma}
By theorem~\ref{theorem:proj.plane.diffeo}, it is therefore diffeomorphic to its model.
\begin{proof}
Cut loci intersect transversely, and by our cohomology calculation they do so at a unique point, so they are the lines of a unique smooth projective plane structure.
\end{proof}

\section{Projective spaces}

There is a more general concept of \emph{projective space} \cite{Beutelspacher/Rosenbaum:1998}: an incidence geometry \((P,L,F)\) so that
\begin{enumerate}
\item Any two points lie on a unique line.
\item (Veblen--Young): any line passing through two sides of a triangle passes through the third as well.
\item Any line contains at least three points.
\item There are at least two lines.
\end{enumerate}
A \emph{subspace} \(S \subset P\) is a set of points so that for any two points \(p, q \in S\), all points of the line \(pq\) lie in \(S\).
The \emph{span} of a set of points is the smallest subspace containing it.
The \emph{dimension} of a projective space is \(n\) if \(P\) is the span of \(n+1\) points, but not of \(n\) 
points.
Every projective space of dimension \(n \ge 3\) is isomorphic to the projective space \(\mathbb{P}^n_k\) over a division ring \(k\), with \(k\) uniquely determined up to isomorphism \cite{Beutelspacher/Rosenbaum:1998} p. 78 theorem 2.7.1, p. 117 theorem 3.4.2.
It is not clear how to define \emph{smooth} projective spaces.

There are two natural definitions of lines on a Blaschke manifold: in the first natural definition, a \emph{line}\label{def:line} in a Blaschke manifold is the set of points of maximal distance from the set of points of maximal distance from two distinct points.
Axioms (1), (3) and (4) are satisfied.
Looking at tangent and normal spaces, we see that lines are smooth submanifolds.
It is not known whether Blaschke manifolds are projective spaces, i.e. whether lines satisfy the Veblen--Young axiom.

In the second natural definition, a \emph{line} in a Blaschke manifold is the union of all geodesics between two maximally distant points.
Axioms (2) and (3) are satisfied, but the other axioms are unknown.

\section{Total geodesy}

\begin{lemma}
Any compact connected totally geodesic submanifold of a Blaschke manifold is itself a Blaschke manifold of the same diameter.
\end{lemma}
\begin{proof}
Recall that all geodesics are periodic in any Blaschke manifold \(M\), with period twice the diameter.
A totally geodesic submanifold \(M' \subset M\) is a submanifold whose geodesics (in the induced metric) are also geodesics of \(M\) that happen to lie in \(M'\).
Let \(D,D'\) be the diameters of \(M,M'\).
The geodesics in \(M\) minimize up to length \(D\).
Therefore the geodesics in \(M'\) also minimize to at least length \(D\).
They can't minimize beyond length \(D'\) so \(D \le D'\).
Take any two points of \(M'\) whose distance in \(M'\) is equal to \(D'\).
They are connected by a minimal geodesic, by the Hopf--Rinow theorem.
Every geodesic has period \(2D\); split it into two halves, and see that it stops minimizing at length at most \(D\).
So \(D' \le D\).
The geodesics in \(M'\) minimize distance up to the diameter, i.e. \(M'\) is Blaschke.
\end{proof}

In particular, the largest dimensional flat totally geodesic torus is of dimension 1: a closed geodesic.
Every compact Riemannian manifold contains a periodic geodesic, by a theorem of Lyusternik and Fet \cite{Bott:1982}.
The product of two compact Riemannian manifolds contains a product of geodesics, a totally geodesic torus: no Blaschke manifold is a product.

The cut loci of points of a Blaschke manifold are compact connected submanifolds; we don't know whether they are totally geodesic.
Any totally geodesic submanifold is determined by any one its tangent planes.
The intersection of any two totally geodesics submanifolds is a submanifold of ``expected'' dimension, i.e. with dimension equal to the dimension of the intersections of their tangent planes.
If all cut loci are totally geodesic, then the cut loci of the cut loci are given by intersecting, i.e. if the diameter of \(M\) is \(D\), and we let \(pM\) be the cut locus of \(p\) in \(M\), then \(q(pM)=p(qM)=(pM) \cap (qM)\) for any two points \(p, q \in M\) with \(q \in pM\).
In particular, the cut loci of the cut loci are totally geodesic too.
So we have an induction strategy for proving the Blaschke conjecture for Blaschke manifolds whose cut loci are totally geodesic.

\begin{theorem}%
[Rovenskii and Toponogov \cite{Rovenskii:1998,Rovenskii/Toponogov:2008}]%
\label{theorem:RT}
Suppose that \(M\) is a compact connected Riemannian manifold of dimension \(kn\), where \(k= 2,4\) or \(8\).
If \(k=8\), we also suppose that \(M\) is 16-dimensional.
Suppose further that
\begin{enumerate}
\item for every point \(m \in M\) and tangent vector \(v \in T_m M\), there is a \(k\)-dimensional totally geodesic submanifold \(S_v \subset M\) isometric to the unit Euclidean sphere tangent to \(v\) at \(m\) and
\item for any \(v_1, v_2 \in T_m M\), the submanifolds \(S_{v_1}, S_{v_2}\) are either equal or intersect only at \(m\).
\end{enumerate}
Let \(M_0\) be the compact rank 1 symmetric space of the dimension as \(M\) and with lines of dimension \(k\), scaled so that the lines of \(M_0\) are isometric to unit Euclidean spheres.
Then \(M\) has volume at least that of \(M_0\), and equality occurs just when \(M\) is isometric to \(M_0\).
\end{theorem}

The proof is again by use of Jacobi vector fields along radial geodesics, in geodesic normal coordinates, and has much in common with \cite{Chavel:2006}  pp. 319--331, but splitting the Jacobi fields into those tangent and those normal to the totally geodesic submanifolds, and using the appearance of the known curvature tensor in the Jacobi equation for the tangent Jacobi fields, so that these Jacobi fields are also explicitly known.

\begin{corollary}\label{corollary:RT}
If all lines (in either sense; see page~\pageref{def:line}) of a Blaschke manifold are totally geodesic, then the Blaschke manifold is isometric to its model.
\end{corollary}
\begin{proof}
Theorem~\vref{theorem:RT} proves that the volume of the Blaschke manifold agrees with that of its model just when they are isometric.
In section~\vref{section:Volume.and.symplectic}, we saw that the volume of any Blaschke manifold is the same as that of its model.
\end{proof}

A geodesic on a Riemannian manifold is \emph{taut} if, for almost every point \(m_0\) of the manifold, the distance function \(m \mapsto d\of{m,m_0}\) restricts to the geodesic to have two critical points.

\begin{theorem}%
[Hebda \cite{Hebda:1981}]
If all geodesics of a Blaschke manifold are taut, then the cut loci of the Blaschke manifold are totally geodesic.
\end{theorem}
\begin{proof}
Take two points \(q,r\) in the same cut loci, the cut locus of some point \(p\).
The distance from \(p\) is maximal at \(q\) and \(r\), and minimal somewhere along a geodesic \(\gamma\) through \(q\) and \(r\).
Since our Blaschke manifold \(M\) is taut, a segment of \(\gamma\) containing both \(q\) and \(r\) lies at maximal distance from \(p\).
So the shape operator at \(q\) of the cut locus of \(p\) is null in the direction tangent to \(\gamma\).
Since this occurs for every point \(r\) in the cut locus of \(p\), taking all \(r\) close to \(q\) in that cut locus, we find that the shape operator is null in all directions through \(q\), so the cut locus is totally geodesic.
\end{proof}

\begin{corollary}%
[Hebda \cite{Hebda:1981}]
If all cut loci of a Blaschke manifold are totally geodesic, then the Blaschke manifold is isometric to its model.
In particular, any Blaschke manifold whose geodesics are taut is isometric to its model.
\end{corollary}
\begin{proof}
Recall that the lines in the first sense are the intersections of cut loci, so totally geodesic, so we apply corollary~\vref{corollary:RT}.
\end{proof}

\section{Exotic spheres and projective spaces}

The proof that Blaschke homology spheres are isometric to their models, as described in section~\vref{section:Volume.and.Jacobi}, uses curvature estimates.
The curvature of the sphere is very peculiar, so it is not surprising that there are no exotic Blaschke homology spheres.
On the other hand, the proof above that a Blaschke manifold is a cohomology sphere or cohomology projective space works identically for any compact Riemannian manifold \(M\) with point \(p\) so that all geodesics through \(p\) are embedded circles of equal length; these are known as \(SC^m\) manifolds \cite{Besse:1978} p. 1.
There is an exotic smooth structure on the 10-dimensional sphere bearing an \(SC^m\) metric \cite{Besse:1978} appendix C; there is even one bearing a metric which is \(SC^m\) at two points, with a foliation by isoparametric hypersurfaces \cite{Qian/Tang:2015} p. 626 theorem 5.1.
Cohomology projective spaces are classified up to diffeomorphism \cite{Kramer/Stolz:2007}.
In particular, the Eells--Kuiper cohomology quaternionic projective spaces \cite{Eells/Kuiper:1962} admit \(SC^m\) metrics \cite{Tang/Zhang:2014}.
It is not known whether any other cohomology projective spaces or cohomology spheres admit \(SC^m\) metrics.

\begin{theorem}%
[Weinstein \cite{Besse:1978} appendix C]
A compact connected manifold \(M\) is the union of the standard Euclidean closed ball and of
a fiber bundle of Euclidean closed balls, glued together by a diffeomorphism of their boundaries, just when there is a Riemannian metric on \(M\) and a point \(p \in M\) so that, for every point \(q \in M\), the set of unit vectors \(u \in T_q M\) tangent to minimal geodesics from \(q\) to \(p\) is a great sphere in the unit sphere in \(T_q M\).
\end{theorem}

\section{Tools we might hope to use}
It is unlikely that the tricks above that produce diffeomorphism, homeomorphism or homotopy equivalence with the model can yield isometry; they forget too much geometry.
A compact connected Riemannian manifold (or length space) with a large number of Blaschke points seems likely to be a Blaschke manifold.
In a Blaschke manifold, the cut locus of any point is a submanifold generating the cohomology ring, with intersections governed by the cohomology ring as in the model, so we might use integral geometry.
The spectra of compact rank 1 symmetric spaces are peculiar, and related to their peculiar geodesic behaviour: we might use spectral geometry.
The compact rank 1 symmetric spaces are characterized by their diameter, upper bound on sectional curvature and lower bound on either injectivity radius \cite{Rovenskii:1998} or distance to conjugate points \cite{Shankar/Spatzier/Wilking:2005}.
A Blaschke manifold, rescaled to have the same diameter as its model, with sectional curvature bounded from below by the minimum sectional curvature of its model, is isometric to its model \cite{Su2015}.
Blaschke manifolds are not known to have bounds on any type of curvature, besides the total scalar curvature; see theorem~\vref{theorem:scalar}.

A \emph{projective connection} on a manifold is an equivalence class of affine connection, with two being equivalent when they have the same geodesics as unparameterized curves.
Lebrun and Mason \cite{LeBrun/Mason:2002,LeBrun/Mason:2010} gave an explicit description of the projective connections on \(S^2\) whose geodesics are closed.
It is unclear how to state an analogue of the Blaschke conjecture for projective connections in all dimensions.
An analogue of the sphere: find the projective connections so that all of the geodesics travelling out of any one point intersect at some other point.
We might also insist that all geodesics be periodic.
More generally, we might like to find the projective connections so that the geodesics travelling out of any one point meet only along a smooth submanifold of some particular codimension, and perhaps also be periodic.

A Riemannian manifold is \emph{harmonic} if near each point there is a nonconstant harmonic function depending only on the distance from that point.
Every compact simply connected harmonic manifold is a compact rank 1 symmetric space \cite{Ballmann:2014,Ranjan:2000,Zsabo:1990}, so we might try to construct some harmonic functions.
Every harmonic manifold \(M\) is isometrically immersed into a Euclidean sphere by mapping each point \(m \in M\) to the values \(\pr{f_1(m), f_2(m), \dots, f_s(m)}\) of an orthonormal basis of \(\lambda\)-eigenfunctions \(f_1, f_2, \dots, f_s\) of the Laplacian of \(M\), for any choice of nonzero eigenvalue \(\lambda\) \cite{Ballmann:2014} p. 30 theorem 6.2; we might look for an analoguous map on a Blaschke manifold.

\nocite{*}
\bibliographystyle{amsplain}
\bibliography{mckay-blaschke}
\end{document}